%% file: main.tex
\documentclass{amsart}
\usepackage[utf8]{inputenc}
\usepackage{biblatex}
\input{includes.tex}

\begin{document}
 
\title{Indistinguishability of cells for the ideal Poisson Voronoi tessellation}
\author{Sam Mellick}
\date{\today}
%\subjclass{Primary: XXX; Secondary: XXX}

\begin{abstract}
In this note, we resolve a question of D'Achille, Curien, Enriquez, Lyons, and Ünel \cite{IPVTlyons} by showing that the cells of the ideal Poisson Voronoi tessellation are indistinguishable. This follows from an application of the Howe-Moore theorem and a theorem of Meyerovitch about the nonexistence of thinnings of the Poisson point process. We also give an alternative proof of Meyerovitch's theorem. 
\end{abstract}

\maketitle

\section{Introduction}

The Poisson Voronoi tessellation is a fundamental object of study in stochastic geometry. Recently there has been interest in the following natural question: how does the Poisson Voronoi tessellation behave as the intensity of the underlying Poisson point process tends to zero?

In Euclidean space, the answer is straightforward: the limiting tessellation is trivial (that is, it is just one cell consisting of the entire space). However, on other symmetric spaces there is a nontrivial limiting object referred to as the \emph{Ideal Poisson Voronoi Tessellation (IPVT)}. This object was proposed in \cite{BCP} for the hyperbolic plane and investigated thoroughly in \cite{IPVTlyons} for hyperbolic space of arbitrary dimension. The discrete analogue of the model (on a $d$-regular tree) was also studied earlier in the PhD thesis \cite{Sandeep}. Inspired by \cite{BCP} and independently of \cite{IPVTlyons}, the model was also studied in \cite{FMW} for more general symmetric spaces, where an application to measured group theory was found. The limit tessellation consists of infinite volume cells. In higher rank symmetric spaces, the tessellation has a quite remarkable property: not only do every pair of cells of the tessellation touch, but they do so along an unbounded ``wall''. This is the key property in the proof \cite{FMW} that higher rank semisimple Lie groups have fixed price one, an important fact in the resolution of conjectures of Gaboriau\cite{GaboriauCout} and Abért, Gelander, and Nikolov\cite{AGN}.

Question 7.8 of \cite{IPVTlyons} asks if the cells of the IPVT on hyperbolic space are \emph{indistinguishable}. This would mean that there is no measurable, equivariant, and deterministic way to select some of the cells in a nontrivial fashion (that is, other than selecting all of the cells, or none of the cells). Indistinguishability is a property that exists in some percolation models\cite{LyonsSchrammIndistinguishable} and has some tremendous applications in measured group theory \cite{GaboriauLyons}. 

In this paper we show that the IPVT of a general symmetric space (in particular, for hyperbolic space $\HH^d$) has indistinguishable cells:

\begin{thm}\label{SellTheorem}
    Let $G$ be a real semisimple Lie group, and $K < G$ its unique (up to conjugacy) maximal compact subgroup. Denote by $X = G/K$ the associated symmetric space. Then the ideal Poisson Voronoi tessellation of $X$ has the property that its cells are indistinguishable. The same is true for the IPVT on a product of regular trees.
\end{thm}

Theorem \ref{SellTheorem} follows from a result of Meyerovitch\cite{Meyerovitch} and an application of the Howe-Moore theorem. We include a streamlined proof of Meyerovitch's result, making use of the \emph{Mecke equation}, a characterisation of the Poisson point process.

\begin{thm}[Meyerovitch]\label{MeyerovitchTheorem}
Let $(X, m)$ be a measure space with $m$ a $\sigma$-finite measure. Suppose $G$ is a group acting on $X$ and preserving $m$. If the action $G \acts (X \times \MM(X), m \otimes \mu)$ is ergodic (where $\mu$ is the distribution of the Poisson point process on $X$ with intensity $m$), then the Poisson point process on $X$ admits no nontrivial thinnings.
\end{thm}

\section*{Acknowledgements}

The author acknowledges the support from the Dioscuri programme initiated by the Max Planck Society, jointly managed with the National Science Centre in Poland, and mutually funded by Polish the Ministry of Education and Science and the German Federal Ministry of Education and Research.

The author would also like to thank Russ Lyons, Mikolaj Fraczyk, Antoine Poulin, and Ben Hayes for comments that improved the readability of the paper.

\section{Background}

We now review the necessary background on the Poisson point process. For a self-contained exposition using the same notation, see \cite{AbertMellick}. For an introduction to the Poisson point process, see \cite{Kingman}, and for a thorough treatment see \cite{LastPenrose} and \cite{Baccelli}. We also recommend the reference works \cite{DVJ1} and \cite{DVJ2}.

Let $X$ be a complete and separable metric space, and $m$ a $\sigma$-finite Borel measure on $X$ with no atoms. The typical case we are interested in is the following: let $G$ denote a locally compact, second countable, and unimodular group. If $H < G$ is a closed and unimodular subgroup, then there exists a unique (up to scaling) $G$-invariant measure $\lambda$ on the quotient space $G/H$. An important subcase here is when $H = 1$, and so we are just taking a group and its Haar measure.

We denote by $\MM = \MM(X)$ the \emph{configuration space}:
\[
    \MM = \{ \omega \subset X \mid \omega \text{ is locally finite} \}.
\]

We equip the configuration space with the Borel structure such that the ``point counting functionals'' $\omega \mapsto \abs{\omega \cap A}$ are measurable for all Borel subsets $A \subseteq X$.  

A \emph{point process} on $X$ is a random element $\Pi$ of the configuration space $\MM$. In this note, we will always have a continuous action $G \acts X$ of a group which preserves $m$. We say that the point process $\Pi$ is \emph{invariant} if its distribution is $G$-invariant. We say that a point process $\Pi$ on $X$ is \emph{Poisson} (with intensity $m$) if it satisfies the following two conditions
\begin{enumerate}
    \item For all Borel $A \subseteq X$, the random variable $\abs{\Pi \cap A}$ is Poisson distributed with parameter $m(A)$. 
    \item For all Borel $A, B \subseteq X$, if $A$ and $B$ are disjoint then the random variables $\abs{\Pi \cap A}$ and $\abs{\Pi \cap B}$ are independent.
\end{enumerate}

Throughout this note, $\mu$ will denote the distribution of the Poisson point process being discussed.

Recall that a random variable $Z$ is \emph{Poisson with parameter $t > 0$} (in short, $Z \sim \mathtt{Pois}(t)$) if
\[
    \PP[Z = k] = e^{-t} \frac{t^k}{k!}.
\]

We will make use of the following characterisation of the Poisson point process, referred to in the literature as either the \emph{Mecke equation}\cite{LastPenrose} or the \emph{Slivnyak-Mecke theorem}\cite{Baccelli}:

\begin{thm}[Mecke equation, see for example Theorem 4.1 of \cite{LastPenrose}]
    A point process $\Pi$ on $X$ is Poisson with intensity $m$ if and only if for all measurable functions $f : X \times \MM \to \RR_{\geq 0}$ we have
    \[
    \EE\left[ \sum_{x \in \Pi} f(x, \Pi) \right] = \int_{X} \EE[f(x, \Pi \cup \{x\})] dm(x).
    \]
\end{thm}

\begin{remark}
The Poisson distribution can be characterised in the following way: a random variable $Z \in \NN$ has the $\mathtt{Pois}(t)$ distribution for $t > 0$ if and only if
\[
    k \PP[Z = k] = t \PP[Z = k-1] \text{ for all } k \in \NN.
\]

The equation clearly follows from the explicit form of the density function for the Poisson, and one can show that the implied recurrence admits a unique solution.

Now consider the Mecke equation for the specific function $\1[x \in A]\1[\abs{\Pi \cap A} = k]$, where $A \subseteq X$ has finite $m$-measure and $k \in \NN$. The lefthand side of the equation has the form
\[
   \EE\left[ \sum_{x \in \Pi} f(x, \Pi) \right] = \EE[k.\1[\abs{\Pi \cap A} = k]] = k \PP[\abs{\Pi \cap A} = k],
\]
whilst the righthand side of the equation has the form
\[
    \int_{X} \EE[f(x, \Pi \cup \{x\})] dm(x) = \int_{A} \PP[\abs{(\Pi \cup \{x\}) \cap A} = k] dm(x) = m(A) \PP[\abs{\Pi \cap A} = k-1].
\]

In this way we see that $\Pi$ satisfies the Mecke equation for all functions of that form if and only if the random variables $\Pi \cap A$ are Poisson with parameter $m(A)$, which is the first requirement for $\Pi$ to be the Poisson point process with intensity $m$. In fact, it is known due to Rényi that this already implies that $\Pi$ is the Poisson point process with intensity $m$, see for example Section 6.3 of \cite{LastPenrose}.

\end{remark}

\begin{defn}
Let $\Pi$ be an invariant point process on $X$. A \emph{thinning} of $\Pi$ is a measurable and equivariant map $\theta : \MM \to \MM$ such that $\theta(\Pi) \subseteq \Pi$ almost surely. If $\theta(\Pi) = \Pi$ almost surely, then we say that the thinning is \emph{full}, whilst if $\theta(\Pi) = \empt$ almost surely, then we say that the thinning is \emph{empty}.

We say that a point process is \emph{indistinguishable} if it admits no nontrivial (that is, neither full nor empty) thinning.
\end{defn}

Theorem \ref{MeyerovitchTheorem} states that, in great generality, there are \emph{no} nontrivial thinnings of the Poisson point process. To illustrate this fact, we include some \emph{nonexamples} to show the types of constructions that do not work.

\begin{example}
    When $(X, m) = (G, \lambda)$, then there is a left-invariant proper metric $d$ on $G$. We may define a family of thinnings $\theta_r$ for $r > 0$ by
    \[
        \theta_r(\Pi) = \{ g \in \Pi \mid d(g, h) \geq r \text{ for all } h \in \Pi \setminus \{g\} \},
    \]
    that is, $\theta_r(\Pi)$ consists of all $r$-isolated points. This thinning may be trivial depending on $\Pi$, but it is always nontrivial for the Poisson point process.
\end{example}

\begin{example}
    For a nonexample, we may take ``independent thinning''. For each point of $\Pi$, independently toss a biased coin with heads probability $p$ and tails probability $1-p$, and let $\theta(\Pi)$ be the points whose coin came up heads. Then $\theta$ is \emph{not} a thinning in the above sense, as it is not a factor of the points of $\Pi$ alone (one requires the additional randomness of the coin tosses). It is a thinning of the ``IID marking'' $[0,1]^\Pi$, however. 

    The IID marking of the Poisson point process may be viewed\footnote{See for example Theorem 5.6 of \cite{LastPenrose}.} as the Poisson point process on $G \times [0,1]$, with $G$ acting trivially on the second coordinate. So this does not fall into the framework of Theorem \ref{MeyerovitchTheorem} (as it is not an ergodic action), and so we should not be surprised that it admits nontrivial thinnings.
\end{example}

\begin{example}
    In \cite{DetThinning}, it is shown that the Poisson point process $\Pi$ on $\RR^n$ can be thinned to produce the Poisson point process of any lower intensity. This was later extended in \cite{Amanda} to compactly generated groups. The reverse problem -- is it possible to \emph{add} points to a Poisson point process and get a Poisson point process of higher intensity -- has also been studied\cite{PoissonThickening}.
\end{example}

\begin{example}
Consider the case of a homogeneous space $(X, m) = (G/H, \lambda)$. Then $G/H$ is metrisable (see Section 8.14 of \cite{HewittRoss}), and we may try to mimic the definition above to construct the $r$-isolated points $\theta_r(\Pi)$. There is no reason to expect that $\theta_r$ will be equivariant. This is because $G/H$ admits a proper left-invariant metric essentially only when $H$ is compact or normal\cite{ADMetric}. 
\end{example}

% \begin{example}
% Consider again the case of a homogeneous space $(X, m) = (G/H, \lambda)$. One may take a \emph{transversal} $T \subseteq \MM$, that is, a subset which meets every $G$ orbit exactly once. One can attempt to use $T$ to construct thinnings: for each $t \in T$, enumerate the configuration, and select (say) the points enumerated by an even number. One can then extend this to a map $\theta : \MM \to \MM$ by equivariance. This appears to be a thinning, except it does not satisfy the measurability assumption: the transversal will \emph{not} be measurable unless $G/H$ is compact
% \end{example}

Recall that a (possibly infinite) measure preserving action $G \acts (X, \mu)$ is \emph{ergodic} if every $G$-invariant measurable subset $A \subseteq X$ (that is, $GA = A$) is $\mu$-null or $\mu$-conull. If $G \acts (X, \mu)$ is a probability measure preserving action, then it is said to be \emph{mixing} if for all measurable $A, B \subseteq X$, we have
\[
    \lim_{g \to \infty} \mu(gA \cap B) = \mu(A) \mu(B).
\]

\section{The ideal Poisson Voronoi tessellation and ergodicity}

For this section, $G$ denotes a real semisimple Lie group or $\Aut(T_{d_1}) 
\times \cdots \times \Aut(T_{d_k})$, the product of automorphism groups of at least two $d_i$-regular trees, with $d_i \geq 3$. The ideal Poisson Voronoi tessellation (IPVT) is an invariant random tessellation of the associated symmetric space\footnote{Here, $K < G$ is the unique (up to conjugacy) maximal compact subgroup. For the case where $G$ is a product of automorphism groups of trees, this is the product of the stabiliser of a point in each tree.} $X = G/K$. In essence, the IPVT is constructed by taking the Poisson Voronoi tessellation on the space and letting the intensity tend to zero. In \cite{IPVTlyons} this is done for $\HH^d$ in the literal sense: they prove weak convergence of the tessellations themselves. An explicit description of the limit tessellation is provided in Section 3 of the paper, and one can easily sample from the IPVT (in a finite volume subset of $\HH^d$), and many beautiful simulations of the process are illustrated in \cite{IPVTlyons}.

In \cite{FMW}, the approach to studying the IPVT is via \emph{generalised} Voronoi tessellations instead, which we now discuss.

Let $X$ denote the symmetric space associated to $G$. Given a countable subset $\omega \subset C(X; \RR)$ of continuous functions on $X$, one can associate to it its generalised Voronoi tessellation 
\[
    \mathscr{V}(\omega) = \{ V_\omega(f) \}_{f \in \omega}, \text{ where } V_\omega(f) = \{ y \in X \mid f(y) \leq f'(y) \text{ for all } f' \in \omega \}.
\]
A priori, this is merely an ensemble of closed sets. However, if $\omega$ comes from a reasonably well behaved class of functions then $\mathscr{V}(\omega)$ will form a locally finite cover of $X$. This is the case, for example, if the family of functions is of the form $\{d(x, \bullet) \mid x \in \omega \}$, where $\omega \subset X$ is a locally finite subset, in which case $\mathscr{V}(\omega)$ is simply the usual Voronoi tessellation of $X$ associated to $\omega$. 

In \cite{FMW}, one takes the Poisson point process $\Pi_t$ of intensity $t$ on $G$, and looks at some family of functions $\Phi_t(\Pi_t)$. The functions are essentially those that measure the distance to the points of $\Pi_t$, but suitably normalised. It is shown that there is a unique weak limit $\Upsilon$ of this invariant random family of functions on $X$. Note that $\Upsilon$ is itself a Poisson point process on $C(X;\RR)$ with respect to a Borel measure. The generalised Voronoi tessellation of $\Upsilon$ is what is referred to as the IPVT. Strictly speaking, it is not shown that the Poisson Voronoi tessellations themselves weakly converge to the IPVT. For the purposes of the paper \cite{FMW} (namely, proving the fixed price one property of $G$), this was not required. 

We now give an alternative description of the IPVT as a $G$-action. Recall that $G$ admits a unique (up to conjugacy) \emph{minimal parabolic subgroup} $P < G$. For semisimple Lie groups this is standard, whilst for products of trees it is simply the product of the stabiliser of an end in each factor. The subgroup $P$ is closed, but not unimodular, and thus admits a nontrivial modular homomorphism $\Delta : P \to \RR_{\geq 0}$. Additionally, there is a nontrivial Radon-Nikodym cocycle $c : G \times G/P \to \RR_{\geq 0}$. We may then take the action $G \acts G/P \times \RR_{\geq 0}$,
\[
    g(xP, t) = (gxP, c(g, xP)t),
\]
and note that this action preserves a unique (up to scaling) measure $m$. 

Finally, there is\footnote{For an explicit description, see Sections 4.3 and 5.1 of \cite{FMW}.} a $G$-equivariant map $\beta : G/P \times \RR_{\geq 0} \to C(X; \RR)$, where $G$ acts on $C(X;\RR)$ by the shift map. 

Let $\Pi$ denote the Poisson point process on $Z := G/P \times \RR_{\geq 0}$ with intensity measure $m$. Then the IPVT is the generalised Voronoi tessellation associated to the family of functions $\beta(\Pi)$. 

In \cite{IPVTlyons}, a notion of  indistinguishability of the cells of the IPVT is introduced, and in Question 7.8 of that paper it is shown to be equivalent to proving that the Poisson point process on $(Z, m)$ admits no nontrivial thinnings. We prove this in Section \ref{indistSect} of this paper, and the proof will require some ergodic theoretic statements that we now prove. 

\begin{lem}\label{BaseErgodicity}
    The action $G \acts (G/P \times \RR_{\geq 0}, m)$ is ergodic.
\end{lem}

The above follows from surjectivity of the modular homomorphism $\Delta$, and holds more generally (see the remarks after Definition 4.2.21 in \cite{Zimmer}).

Let $\mu$ denote the law of the Poisson point process on $Z$ with intensity measure $m$. 

\begin{lem}\label{MixingLemma}
    The action $G \acts (\MM(Z), \mu)$ is ergodic (in fact, mixing).
\end{lem}

\begin{proof}[Sketch of proof]
    This is proved in a similar way to establishing mixing for Bernoulli shifts, as for example in Proposition 7.3 of \cite{LyonsPeres}. There are minor technical difficulties in extending the proof to Poisson point processes, see Proposition 8.13 of \cite{LastPenrose}. The key task is to show that for $A, B \subseteq Z$ of finite volume, we have
    \[
        \lim_{n \to \infty} m(gA \cap B) = m(A)m(B),
    \]
    which follows from the Howe-Moore property (see Item (3) of Theorem 3.8 of \cite{FMW}).
\end{proof}

\begin{lem}\label{ErgodicityLemma}
    The action $G \acts (Z \times \MM(Z), m \otimes \mu)$ is ergodic. 
\end{lem}

\begin{proof}
    In \cite{Schmidt}, a notion of \emph{mild mixing} is introduced (a weakening of the notion of mixing). It follows from Proposition 2.3 of that paper that the product of a mildly mixing action and a properly ergodic\footnote{Recall that a nonsingular action is \emph{properly ergodic} if it is ergodic, but not concentrated on a single orbit, that is, it is not essentially transitive.} infinite measure preserving action is ergodic. Thus, the action $G \acts (Z \times \MM(Z), m \otimes \mu)$ is ergodic if $G \acts (Z, m)$ is properly ergodic.

    If $G \acts (Z, m)$ is \emph{not} properly ergodic, then there exists a closed, unimodular subgroup $H < G$ such that $G \acts (Z, m)$ is ismorphic (as an action of measure spaces) to the action of $G \acts G/H$ with the invariant measure on the quotient. We distinguish two cases, as to whether or not $H$ is compact. Compactness of $H$ leads to a contradiction, and in the case of noncompactness we prove ergodicity of $G \acts (Z \times \MM(Z), m \otimes \mu)$ directly.

    In Theorem 3.8 (3) of \cite{FMW}, it is shown that the action $G \acts (Z, m)$ is \emph{doubly recurrent}. This immediately implies that the action $G \acts (Z, m)$ is \emph{conservative}, that is, for all measurable $A \subseteq Z$ and $m$ almost every $a \in A$, the set of return times $\{g \in G \mid ga \in A \}$ is unbounded. 

   Observe that for a transitive action $G \acts G/H$, the set of return times of a point $gH$ is simply the conjugate subgroup $gHg^{-1}$. It follows from conservativity that the subgroup $H$ cannot be compact. 

   Now $H$ is noncompact, and hence too all of its conjugates $gHg^{-1}$. Since mixing passes to noncompact subgroups, we get that the actions $gHg^{-1} \acts (\MM(G/H), \mu)$ are all mixing (and hence ergodic). Suppose that $A \subseteq G/H \times \MM(G/H)$ is $G$-invariant. Then for all $(gH, \omega) \in A$, we have that $(gH, \gamma.\omega) \in A$ for all $\gamma \in gHg^{-1}$. It follows from mixing then that for all $gH \in G/H$,
   \[
        \mu(\{\omega \in \MM(G/H) \mid (gH, \omega) \in A \}) = \text{0 or 1}.
   \]
   This is a $G$-invariant function, and so it follows that $A$ is either null or conull, as desired.
\end{proof}

\section{Indistinguishability }\label{indistSect}

We now state and prove the main characterisation of fullness and emptiness of thinnings that we will use for our alternative proof of Theorem \ref{MeyerovitchTheorem}:

\begin{lem}\label{ThinLemma}
Let $\Pi$ be the Poisson point process on $X$ and $\theta(\Pi)$ a thinning of it.

The thinning is full if and only if for $m$ almost every $x \in X$, we have $\PP[x \in \theta(\Pi \cup \{x\})] = 1$. 

The thinning is empty if and only if for $m$ almost every $x \in X$, we have $\PP[x \in \theta(\Pi \cup \{x\})] = 0$. 

In the case of a homogeneous space $(X, m) = (G/H, \lambda)$, a thinning is full or empty if and only if $\PP[H \in \theta(\Pi \cup \{H\})]$ is $1$ or $0$ respectively.

\end{lem}

A point process $\Pi$ on a $G$-space $X$ is \emph{indistinguishable} if it admits no nontrivial (non empty, non full) thinning.

\begin{proof}[Proof of Lemma \ref{ThinLemma}]

We apply the Mecke equation with the function $f(x, \Pi) = \1[x \in A]\1[x \in \theta(\Pi)]$, where $A \subseteq X$ has finite volume. Then the Mecke equation states that
\[
\EE\abs{A \cap \theta(\Pi)} = \int_A \PP[x \in \theta(\Pi \cup \{x\})]d\lambda(x).
\]
If the thinning is full, then the lefthand side is simply $m(A)$, and hence $\PP[x \in \theta(\Pi \cup \{x\})] = 1$ for all $m$ almost every $x \in A$. By choosing an exhaustion of $X$ we can extend this to $m$ almost every $x \in X$. A similar argument shows that if the thinning is empty, then $\PP[x \in \theta(\Pi \cup \{x\})] = 0$ for $m$ almost every $x \in X$.

Conversely, if $\PP[x \in \theta(\Pi \cup \{x\})] = 1$ for $m$ almost every $x \in X$, then the righthand side of the equation above is simply $m(A)$. But $\theta(\Pi)$ is a subset of $\Pi$, so we conclude that $(\Pi \setminus \theta(\Pi)) \cap A = \empt$ almost surely. By an exhaustion argument we see that $\theta$ is full. A similar argument shows that if $\PP[x \in \theta(\Pi \cup \{x\})] = 0$ for $m$ almost every $x \in X$, then $\theta$ is empty. 

Finally, for the case of a homogeneous space observe that by distributional invariance of $\Pi$ and equivariance of $\theta$ we have
\[
    \PP[gH \in \theta(\Pi \cup \{gH\})] = \PP[gH \in \theta(g\Pi \cup \{gH\})] = \PP[H \in \theta(\Pi \cup \{H\})],
\]
and so the desired characterisation follows from the previous arguments.
\end{proof}

\begin{proof}[Proof of Theorem \ref{MeyerovitchTheorem}]

Define $A = \{ (x, \omega) \in X \times \MM \mid x \in \theta(\omega \cup \{x\})\}$, and observe that $A$ is $G$-invariant. So by ergodicity, either $(m \otimes \mu)(A) = 0$ or $(m \otimes \mu)(A^c) = 0$. In the former case, we have
\[
    0 = (m \otimes \mu)(A) = \int_X \PP[x \in \theta(\Pi \cup \{x\})]dm(x),
\]
and hence $\PP[x \in \theta(\Pi \cup \{x\})] = 0$ for $m$ almost every $x \in X$. Thus $\theta$ is the empty thinning by Lemma \ref{ThinLemma}. The other case is similar.
\end{proof}

\begin{proof}[Proof of Theorem \ref{SellTheorem}]

As noted in Question 7.8 of \cite{IPVTlyons}, indistinguishability of the cells of the IPVT is equivalent to showing that the Poisson point process on $(Z, m)$ admits no nontrivial thinnings. The proof indistinguishability is thus is a direct application of Theorem \ref{MeyerovitchTheorem} and Lemma \ref{ErgodicityLemma}.
\end{proof}

% \begin{remark}
%     In the above proof, one can bypass Lemma \ref{ErgodicityLemma} and just use mixing of the action $G \acts \MM(G/H)$ for the case of the IPVT for a semisimple Lie group $G$. Simply apply the thinning criterion for a homogeneous space described in Lemma \ref{ThinLemma} and observe that for all $h \in H$, by equivariance of $\theta$ we have
%     \[
%         \{ H \in \theta(\Pi \cup \{H\}) \} = \{ H \in \theta(\Pi \cup h\{H\}) \} = \{ H \in h^{-1}\theta(\Pi \cup \{H\}) \} = \{ H \in \theta(\Pi \cup \{H\}) \},
%     \]
%     and so the event $\{ H \in \theta(\Pi \cup \{H\}) \}$ is $H$-invariant. Now since $H$ is noncompact
% \end{remark}

In percolation theory, indistinguishability admits an ergodic interpretation. For a countable group $\Gamma$, there is a natural subrelation of the orbit equivalence relation of the action $\Gamma \acts \{0,1\}^\Gamma$ referred to as the \emph{cluster equivalence relation} (see \cite{GaboriauLyons}). Indistinguishability of the percolation model corresponds exactly to ergodicity of this subrelation (restricted to the infinite locus). 

After circulation of an early draft of this article, Ben Hayes asked if indistinguishability for the IPVT also admits an ergodic interpretation. Antoine Poulin, in private communication, suggested the following:

\begin{question}\footnote{This question has since been resolved by Kohki Sakamoto, in private communication.}
    Is the restricted rerooting equivalence relation of the IPVT induced on the Palm process of the Poisson point process on $X$ ergodic?
\end{question}

The restricted rerooting relation is discussed in \cite{FMW}, Section 7.1, where it is denoted as $\mathcal{S}$ and referred to simply as the restricted sub-relation. In that paper it is shown to be hyperfinite, which was sufficient for the application to the fixed price one problem. Note that if the IPVT admitted nontrivial thinnings, then the restricted rerooting relation would not be ergodic. 

\printbibliography

\end{document}

%% file: includes.tex
\usepackage[margin=1in]{geometry}  
\usepackage{amsmath}               
\usepackage{amsfonts}              
\usepackage{amsthm}      
\usepackage{amssymb}
\usepackage{bbm}

\usepackage[utf8]{inputenc}
\usepackage{mathtools}

\usepackage{color}   %May be necessary if you want to color links
\usepackage{hyperref}
\hypersetup{
    colorlinks=true, %set true if you want colored links
    linktoc=all,     %set to all if you want both sections and subsections linked
    linkcolor=black,  %choose some color if you want links to stand out
}

\def\acts{\curvearrowright}
\DeclarePairedDelimiter\abs{\lvert}{\rvert}%
\DeclarePairedDelimiter\norm{\lVert}{\rVert}%

\makeatletter
\let\oldabs\abs
\def\abs{\@ifstar{\oldabs}{\oldabs*}}
\let\oldnorm\norm
\def\norm{\@ifstar{\oldnorm}{\oldnorm*}}

\makeatother

\theoremstyle{definition}
\newtheorem{thm}{Theorem}
\newtheorem{lem}[thm]{Lemma}

\newtheorem{defn}[thm]{Definition}
\newtheorem*{example}{Example}

\newtheorem{remark}[thm]{Remark}

\newtheorem{question}[thm]{Question}

\DeclareMathOperator{\Aut}{Aut}

\let\phi\varphi

\let\empt\varnothing

\newcommand{\EE}{\mathbb{E}}      % for Real numbers

\newcommand{\RR}{\mathbb{R}}      % for Real numbers

\newcommand{\PP}{\mathbb{P}}      % for Integers
\newcommand{\MM}{\mathbb{M}}      % for Integers

\newcommand{\NN}{\mathbb{N}}      % for Integers
\newcommand{\HH}{\mathbb{H}}      % for Integers

\newcommand{\1}{\mathbbm{1}}

\newcommand\restr[2]{{% we make the whole thing an ordinary symbol
  \left.\kern-\nulldelimiterspace % automatically resize the bar with \right
  #1 % the function
  \vphantom{\big|} % pretend it's a little taller at normal size
  \right|_{#2} % this is the delimiter
  }}

\usepackage{color}
\usepackage{mathrsfs}

%% file: bibliography.bib
@Article{AbertMellick,
 Author = {Ab{\'e}rt, Mikl{\'o}s and Mellick, Sam},
 Title = {Point processes, cost, and the growth of rank in locally compact groups},
 FJournal = {Israel Journal of Mathematics},
 Journal = {Isr. J. Math.},
 ISSN = {0021-2172},
 Volume = {251},
 Number = {1},
 Pages = {48--155},
 Year = {2022},
 Language = {English},
 DOI = {10.1007/s11856-022-2445-9},
 zbMATH = {7661284},
 Zbl = {1514.37011}
}

@misc{FMW,
      title={Poisson-Voronoi tessellations and fixed price in higher rank}, 
      author={Mikolaj Fraczyk and Sam Mellick and Amanda Wilkens},
      year={2023},
      eprint={2307.01194},
      archivePrefix={arXiv},
      primaryClass={math.GT},
      url={https://arxiv.org/abs/2307.01194}, 
}

@Book{LastPenrose,
 Author = {Last, G{\"u}nter and Penrose, Mathew},
 Title = {Lectures on the {Poisson} process},
 FSeries = {Institute of Mathematical Statistics Textbooks},
 Series = {IMS Textb.},
 Volume = {7},
 ISBN = {978-1-107-45843-7; 978-1-107-08801-6; 978-1-316-10447-7},
 Year = {2018},
 Publisher = {Cambridge: Cambridge University Press},
 Language = {English},
 DOI = {10.1017/9781316104477},
 URL = {semanticscholar.org/paper/f397e902cdd93643ca46c82d0ba7214af2e1975d},
 zbMATH = {6796875},
 Zbl = {1392.60004}
}

@misc{IPVTlyons,
      title={Ideal Poisson-Voronoi tessellations on hyperbolic spaces}, 
      author={Matteo D'Achille and Nicolas Curien and Nathanaël Enriquez and Russell Lyons and Meltem Ünel},
      year={2023},
      eprint={2303.16831},
      archivePrefix={arXiv},
      primaryClass={math.PR},
      url={https://arxiv.org/abs/2303.16831}, 
}

@misc{BCP,
      title={On Cheeger constants of hyperbolic surfaces}, 
      author={Thomas Budzinski and Nicolas Curien and Bram Petri},
      year={2022},
      eprint={2207.00469},
      archivePrefix={arXiv},
      primaryClass={math.GT},
      url={https://arxiv.org/abs/2207.00469}, 
}

@Article{DetThinning,
 Author = {Holroyd, Alexander E. and Lyons, Russell and Soo, Terry},
 Title = {Poisson splitting by factors},
 FJournal = {The Annals of Probability},
 Journal = {Ann. Probab.},
 ISSN = {0091-1798},
 Volume = {39},
 Number = {5},
 Pages = {1938--1982},
 Year = {2011},
 Language = {English},
 DOI = {10.1214/11-AOP651},
 zbMATH = {5987674},
 Zbl = {1277.60087}
}

@Article{Amanda,
 Author = {Wilkens, Amanda},
 Title = {Isomorphisms of {Poisson} systems over locally compact groups},
 FJournal = {The Annals of Probability},
 Journal = {Ann. Probab.},
 ISSN = {0091-1798},
 Volume = {51},
 Number = {6},
 Pages = {2158--2191},
 Year = {2023},
 Language = {English},
 DOI = {10.1214/23-AOP1642},
 zbMATH = {7795619},
 Zbl = {1537.37016}
}

@Article{GaboriauLyons,
 Author = {Gaboriau, Damien and Lyons, Russell},
 Title = {A measurable-group-theoretic solution to von {Neumann}'s problem},
 FJournal = {Inventiones Mathematicae},
 Journal = {Invent. Math.},
 ISSN = {0020-9910},
 Volume = {177},
 Number = {3},
 Pages = {533--540},
 Year = {2009},
 Language = {English},
 DOI = {10.1007/s00222-009-0187-5},
 zbMATH = {5598917},
 Zbl = {1182.43002}
}

@Article{LyonsSchrammIndistinguishable,
 Author = {Lyons, Russell and Schramm, Oded},
 Title = {Indistinguishability of percolation clusters},
 FJournal = {The Annals of Probability},
 Journal = {Ann. Probab.},
 ISSN = {0091-1798},
 Volume = {27},
 Number = {4},
 Pages = {1809--1836},
 Year = {1999},
 Language = {English},
 DOI = {10.1214/aop/1022677549},
 zbMATH = {1496108},
 Zbl = {0960.60013}
}

@Book{Kingman,
 Author = {Kingman, J. F. C.},
 Title = {Poisson processes},
 FSeries = {Oxford Studies in Probability},
 Series = {Oxf. Stud. Probab.},
 Volume = {3},
 ISBN = {0-19-853693-3},
 Year = {1993},
 Publisher = {Oxford: Clarendon Press},
 Language = {English},
 zbMATH = {227027},
 Zbl = {0771.60001}
}

@book{Baccelli,
  TITLE = {{Random Measures, Point Processes, and Stochastic Geometry}},
  AUTHOR = {Baccelli, Fran{\c c}ois and Blaszczyszyn, Bartlomiej and Karray, Mohamed},
  URL = {https://inria.hal.science/hal-02460214},
  PUBLISHER = {{Inria}},
  YEAR = {2024},
  MONTH = Jul,
  HAL_ID = {hal-02460214},
  HAL_VERSION = {v2},
}

@Book{DVJ1,
 Author = {Daley, D. J. and Vere-Jones, D.},
 Title = {An introduction to the theory of point processes. {Vol}. {I}: {Elementary} theory and methods.},
 Edition = {2nd ed.},
 FSeries = {Probability and its Applications},
 Series = {Probab. Appl.},
 ISSN = {2297-0371},
 ISBN = {0-387-95541-0},
 Year = {2003},
 Publisher = {New York, NY: Springer},
 Language = {English},
 DOI = {10.1007/b97277},
 zbMATH = {1855795},
 Zbl = {1026.60061}
}

@Book{DVJ2,
 Author = {Daley, Daryl J. and Vere-Jones, David},
 Title = {An introduction to the theory of point processes. {Vol}. {II}: {General} theory and structure.},
 Edition = {2nd revised and extended ed.},
 FSeries = {Probability and its Applications},
 Series = {Probab. Appl.},
 ISSN = {2297-0371},
 ISBN = {978-0-387-21337-8; 978-0-387-49835-5},
 Year = {2008},
 Publisher = {New York, NY: Springer},
 Language = {English},
 DOI = {10.1007/978-0-387-49835-5},
 zbMATH = {5180190},
 Zbl = {1159.60003}
}

@Book{HewittRoss,
 Author = {Hewitt, E. and Ross, K. A.},
 Title = {Abstract harmonic analysis. {Volume} {I}: {Structure} of topological groups, integration theory, group representations.},
 Edition = {2nd ed.},
 FSeries = {Grundlehren der Mathematischen Wissenschaften},
 Series = {Grundlehren Math. Wiss.},
 ISSN = {0072-7830},
 Volume = {115},
 ISBN = {3-540-94190-8},
 Year = {1994},
 Publisher = {Berlin: Springer-Verlag},
 Language = {English},
 zbMATH = {786924},
 Zbl = {0837.43002}
}

@Article{ADMetric,
 Author = {Anantharaman-Delaroche, Claire},
 Title = {Invariant proper metrics on coset spaces},
 FJournal = {Topology and its Applications},
 Journal = {Topology Appl.},
 ISSN = {0166-8641},
 Volume = {160},
 Number = {3},
 Pages = {546--552},
 Year = {2013},
 Language = {English},
 DOI = {10.1016/j.topol.2013.01.001},
 zbMATH = {6139372},
 Zbl = {1275.22017}
}

@Article{Meyerovitch,
 Author = {Meyerovitch, Tom},
 Title = {Ergodicity of {Poisson} products and applications},
 FJournal = {The Annals of Probability},
 Journal = {Ann. Probab.},
 ISSN = {0091-1798},
 Volume = {41},
 Number = {5},
 Pages = {3181--3200},
 Year = {2013},
 Language = {English},
 DOI = {10.1214/12-AOP824},
 zbMATH = {6226021},
 Zbl = {1279.60061}
}

@Article{PoissonThickening,
 Author = {Gurel-Gurevich, Ori and Peled, Ron},
 Title = {Poisson thickening},
 FJournal = {Israel Journal of Mathematics},
 Journal = {Isr. J. Math.},
 ISSN = {0021-2172},
 Volume = {196},
 Pages = {215--234},
 Year = {2013},
 Language = {English},
 DOI = {10.1007/s11856-012-0181-2},
 zbMATH = {6221772},
 Zbl = {1306.60052}
}

@Article{GaboriauCout,
 Author = {Gaboriau, Damien},
 Title = {Cost of equivalence relations and groups},
 FJournal = {Inventiones Mathematicae},
 Journal = {Invent. Math.},
 ISSN = {0020-9910},
 Volume = {139},
 Number = {1},
 Pages = {41--98},
 Year = {2000},
 Language = {French},
 DOI = {10.1007/s002229900019},
 zbMATH = {1388907},
 Zbl = {0939.28012}
}

@Article{AGN,
 Author = {Abert, Miklos and Gelander, Tsachik and Nikolov, Nikolay},
 Title = {Rank, combinatorial cost, and homology torsion growth in higher rank lattices},
 FJournal = {Duke Mathematical Journal},
 Journal = {Duke Math. J.},
 ISSN = {0012-7094},
 Volume = {166},
 Number = {15},
 Pages = {2925--2964},
 Year = {2017},
 Language = {English},
 DOI = {10.1215/00127094-2017-0020},
 URL = {real.mtak.hu/74284/1/1509.01711v2.pdf},
 zbMATH = {6812212},
 Zbl = {1386.22008}
}

@Book{Zimmer,
 Author = {Zimmer, Robert J.},
 Title = {Ergodic theory and semisimple groups},
 FSeries = {Monographs in Mathematics},
 Series = {Monogr. Math., Basel},
 ISSN = {1017-0480},
 Volume = {81},
 Year = {1984},
 Publisher = {Birkh{\"a}user, Cham},
 Language = {English},
 zbMATH = {3911340},
 Zbl = {0571.58015}
}

@phdthesis{Sandeep,
  title={The low-intensity limit of Bernoulli-Voronoi and Poisson-Voronoi measures},
  author={Bhupatiraju, Sandeep},
  year={2019},
  school={[Bloomington, Ind.]: Indiana University]}
}

@Article{Schmidt,
 Author = {Schmidt, Klaus},
 Title = {Asymptotic properties of unitary representations and mixing},
 FJournal = {Proceedings of the London Mathematical Society. Third Series},
 Journal = {Proc. Lond. Math. Soc. (3)},
 ISSN = {0024-6115},
 Volume = {48},
 Pages = {445--460},
 Year = {1984},
 Language = {English},
 DOI = {10.1112/plms/s3-48.3.445},
 zbMATH = {3857426},
 Zbl = {0539.28010}
}

@Book{LyonsPeres,
 Author = {Lyons, Russell and Peres, Yuval},
 Title = {Probability on trees and networks},
 FSeries = {Cambridge Series in Statistical and Probabilistic Mathematics},
 Series = {Camb. Ser. Stat. Probab. Math.},
 Volume = {42},
 ISBN = {978-1-107-16015-6; 978-1-108-73272-7; 978-1-316-67281-5},
 Year = {2016},
 Publisher = {Cambridge: Cambridge University Press},
 Language = {English},
 DOI = {10.1017/9781316672815},
 zbMATH = {6653785},
 Zbl = {1376.05002}
}
